\documentclass[11pt]{article}

\usepackage{amsthm,amssymb,amsfonts,amsmath}
\usepackage{amscd} 
\usepackage{mathrsfs}
\usepackage[numbers, sort&compress]{natbib}
\usepackage{graphicx}
\usepackage[pagebackref=true]{hyperref}
\usepackage{color}

\providecommand{\R}{}
\providecommand{\Z}{}
\providecommand{\N}{}

\providecommand{\Q}{}
\providecommand{\G}{}

\renewcommand{\R}{\mathbb{R}}
\renewcommand{\Z}{\mathbb{Z}}
\renewcommand{\N}{{\mathbb N}}

\renewcommand{\Q}{\mathbb{Q}}
\renewcommand{\G}{\mathbb{G}}

\providecommand{\ora}[1]{}
\renewcommand{\ora}[1]{\overrightarrow{#1}}

\newtheorem{thm}{Theorem}
\newtheorem{lem}[thm]{Lemma}
\newtheorem{prop}[thm]{Proposition}
\newtheorem{cor}[thm]{Corollary}
\newtheorem{dfn}[thm]{Definition}

\newtheorem{example}[thm]{Example}

\numberwithin{thm}{section}
\numberwithin{equation}{section}
\usepackage{graphicx}
\usepackage{algpseudocode}
\usepackage{amssymb}
\usepackage{changepage}
\usepackage{hyperref}
\begin{document}
\title{A Randomness Test Formalism for Neutral Measures and Beyond}
\author{Jacob Canel\\ Department of Mathematics, Pennsylvania State University, \\ \href{mailto:jmc8684@psu.edu}{jmc8684@psu.edu}}
    
\maketitle
\begin{abstract}
    We discuss in detail Neutral Measures, which are measures which believably generate any real in Cantor space. We intuitively build up the classical notions of randomness, and show that neutral measures do exist, using the language of continuous semimeasures. Finally, we construct some specific randomness tests which allow us to enforce desirable properties on neutral measures. As an application, we show the existence of Gibbs measures for a large class of finite range Hamiltonians on lattice models, and establish their relationship with neutral measures.
\end{abstract}
\section{Introduction}\label{sec:Intro}
In an attempt to reconcile the gap between the intuitive notion of randomness with the rigorous, measure theoretic one, a number of mathematicians have created randomness notions, of varying strength, which comprise null sets defined by some concrete notion of atypical behavior. These notions range from friendly to apply (such as obeying the law of large numbers at some predetermined convergence rate), to impossible to explicitly compute (for example Martin-L\"of randomness, as introduced in \cite{MARTINLOF}).
This kind of randomness test allows one to apply a definition of randomness to \textit{individual reals}: rather than working non-constructively from the assumption that a real has been generated randomly from a measure $\mu$, and inferring statistical facts about that real, one can ask directly if an individual real looks $\mu$-random. This approach is related to the theory of effective (Hausdorff) dimension, in which the approach of looking at individual reals has gained some traction even in applications outside of logic and theoretical computer science \cite{lutz2016}.

In developing randomness notions for non-computable measures, one must ask what information they might be given about the underlying measure. This creates a kind of arms race between measures and randomness notions. More powerful randomness tests require measures to spread out and concentrate among many classes of real, and more sophisticated measures inevitably give tests more computability theoretic power to find small-measure sets and thus identify more reals as non-random.
In his 1973 paper, Levin \cite{levin} ended the conflict in favor of measures, showing via a fixed point argument that there are measures, the so called \textit{neutral measures}, which have no non-random reals. A similar argument was put forth recently by Day and Miller \cite{daymiller}, who study the computational complexity of neutral measures, showcasing their place among the enumeration degrees.
What is, in the opinion of the author, missing from this discussion is the far-reaching consequences of the fixed point argument, which proves more than just the existence of neutral measures. We will show that the existence of Gibbs measures for arbitrary positive temperature, for example, follows from the same argument, and that these among other measures are supported by (absolutely Lipschitz continuous with respect to) any arbitrary neutral measure. 

Section \ref{sec:Rand} contains an as-friendly-as-possible introduction to the theory of randomness, at least those parts needed for later sections. We build up from basic intuitions the machinery we need to ask the relevant questions about neutral measures. We also argue against the existence of neutral measures on $\N^\N$.

Section \ref{sec:Tests} is where we concretely define the notion of randomness test we are using, which coincides in large part with that in \ref{sec:Rand}, but is essentially only reliant on the notion that the test should get stricter when supplied with more information. We discuss which classes of tests can be captured by single, harder test.

Section \ref{sec:Neutral} is devoted to the proof of \ref{thm:mainthm}, in which we largely follow Day and Miller and prove it by means of the Kakutani Fixed Point Theorem. We also discuss an attempt at satisfying the properties of neutral measures which fails in an interesting way, and use this to justify why we need to use a non-constructive fixed point theorem to resolve the problem.

Section \ref{sec:App} contains the bulk of new results. We discuss the existence of Gibbs measures. We show that the dimension distribution of a neutral measure dominates a neutral measure. For all of these, particular tests which enforce the desirable properties are constructed.
\section{The Randomness of an Outcome}\label{sec:Rand}
This section is devoted to a discussion of randomness notions for a few particular measure spaces. It is meant to help build intuition and make heuristic arguments to readers who are not already very familiar with the field.
\subsection{Hypothesis Testing}\label{sec:Flips}
Statistics and hypothesis testing are necessary to underpin modern society and industry, and yet the most fundamental notion (that of a real number being generated according to a specified probability distribution) is not intuitively encoded in the measure theoretic language that expresses it. Given that assumption, they can define with complete clarity the relative likelihood of a particular outcome being generated by a particular distribution or another.\\

In the traditional theory of algorithmic randomness (Martin-L\"{o}f randomness, see Downey and Hirschfeld \cite{DH} for details), the non-randomness of a particular real is established by isolating it in smaller and smaller sets which are algorithmically definable. Non-randomness is defined in terms of quantitatively atypical behavior. One characterization of this behavior is predictability, which brings this notion more into alignment with the everyday definition of randomness.\\

We argue that this seeming disagreement can be rectified if one allows themselves to broaden the definition of an alternative hypothesis. The path of widening the net in this way requires some care to avoid giving ourselves too much power. We propose the following scheme:
\begin{enumerate}
    \item The \textbf{null hypothesis}, against which we are testing, should be a probability measure $\mu$ on $2^\N$.
    \item The \textbf{alternative hypothesis} should be a family of probability measures represented by a common lower bound in a continuous semimeasure $\nu$ (a superadditive functional on $C(2^\N)$, more precision will come later)
    \item The \textbf{experiment} will be a given real $x\in 2^\N$
    \item We will reject the null hypothesis in favor of the alternative hypothesis iff $\liminf_n \frac{\mu([x\upharpoonright_n])}{\nu([x\upharpoonright_n])} = 0$
\end{enumerate}

In plain terms, we will regard the experiment as showing the alternative hypothesis explains the data of $x$ if the probability according to $\mu$ of generating a real around $x$ declines much faster than the probability according to $\nu$ (which really means the probability according to any and all honest-to-goodness probability measures that $\nu$ bounds from below).\\
We want to say that a real $x$ will be $\mu$ non-random if there is any alternative hypothesis for which we would reject $\nu$. However, unless $\mu\{x\} > 0$, we could take $\nu = \delta_x$ and achieve that result. We therefore must limit ourselves as statisticians do in the real world and only use information about the null hypothesis to generate the alternative. The right way to do this turns out to be that the alternative hypothesis should be enumeration reducible to the null. We will define this concept rigorously in the next section, so for now we will restrict ourselves to computable null hypotheses (such as the Lebesgue measure on $2^\N$), which simplifies our new constraint to requiring that the collection 
$$UG(\nu)=\{ (\sigma,q)\in 2^{<\N}\times\Q: \nu([\sigma]) > q\}$$
be a c.e. set.\\

If we require this, then there turns out to be a best-possible alternative hypothesis: one which will cause us to reject the null as often as possible for a c.e. alternative. We can do this via the following trick

\begin{enumerate}
    \item Produce an exhaustive list of the alternate hypotheses we can create, those being semimeasures with lower bounds given by any algorithm (by which we mean Turing machine). Call the $n$-th element of this list $\mu_n$.
    \item Add them up with exponentially decaying weights, as $\mu = \sum_{n\in \N} 2^{-n}\mu_n$.
    \item $\mu$ is thus an alternate hypothesis which will cause us to reject the null hypothesis $\lambda$ whenever we would for any algorithmically constructed alternate hypothesis.
    
\end{enumerate}
For the null hypothesis given by the Lebesgue measure $\lambda$, it turns out that the statement "$x$ causes us to reject the null hypothesis" is the same as "$x$ is not Martin-L\"{o}f random." Indeed, working with this setup produces the same results as Levin's uniform tests of randomness \cite{levin}\cite{daymiller}. We now go into more detail on how to create alternative hypotheses in general.

\subsection{The Enumeration Degrees}

Restricting to computably enumerable alternative hypotheses certainly satisfies the statistical dogma against choosing an alternate hypothesis post hoc, but it does not take full advantage of the information we have prior to performing an experiment: the actual value of the measure which serves as our null hypothesis. We will wish to leverage this data as well, but we have to be careful about how we do so.

Here we begin a review of material about computability theory and the Turing and enumeration degrees, the latter of which were introduced in \cite{enumerationOG}, and expanded upon in a number of works since, such as \cite{enumeration}. For a comprehensive introduction into computability theory in general and how it applies to randomness specifically, a reader is encouraged to consult the relevant sections of  Downey and Hirschfeldt \cite{DH}.

The object of our discussion is Borel measures on the space $2^{\N}$ (although we may later want to give $\N$ the structure of a countable amenable group).
Since we wish to discuss the computable aspects of the theory, we must discuss how these are to be presented. In principle, we should want to know the value of the measure on arbitrary cylinder sets. 
As such, it would not be wholly unreasonable to expect a representation of such a measure to include some encoding of these values, i.e. a function from the cylinder sets to fast converging sequences that go to the measure on that cylinder.
This, as it turns out, is the wrong thing to ask for. Why?

\begin{enumerate}
    \item A measure represented in this way is an infinite join of a countable collection of reals. This is not a well defined operation on the Turing degrees
    \item In particular, such an object might well not have a minimal representative in the Turing degrees, this means that we are unable to conveniently describe a system which has knowledge of the cylinder set values and \textbf{only} those and the things which may be computed from them.
    \item Since we cannot a priori limit our knowledge to something sufficient and necessary for uniquely identifying the measure, our ability to discuss pseudorandomness is greatly diminished.
\end{enumerate}

So what should we want? It is true that we should be able to ascertain the cylinder set values, but explicitly giving their values in a computer-understandable format is giving too much information. We make two observations:

\begin{enumerate}
    \item The positive Borel measures on a compact metric space $X$ are the positive part of the dual of $C(X)$, the continuous functions on $X$
    \item For an individual real, giving an upper-enumeration and a lower-enumeration of it is (computably) equivalent to giving a fast converging sequence to it. This is not true for infinite joins.
\end{enumerate}

Indeed, the appropriate presentation is just exactly this: an enumeration of each of the cylinder set values. More precisely, when we do any steps for which computability is at all demanded, this is what we assume we have access to.

Measures need not have well defined Turing degrees, but they most certainly have well defined \textit{enumeration degrees}.

We will now briefly discuss the enumeration degrees. Given that they are defined in contrast with and in extension of the Turing degrees, we will start there.

\begin{dfn}
    A \textbf{Turing functional} is a Turing machine with two tapes, one assumed to be read-only. The transition function now includes internal states, the read value on the read-write tape, and the read value on the read-only tape in its input.
\end{dfn}

We say that an element $A$ of $2^\N$ computes another element $B$ if there is a Turing functional with $A$ on the read-only tape which, on input $n$, produces the $n$-th bit of $B$ for any $n$.

This definition is intuitive: a computer given the data of $A$ can, using a finite program, produce the data of $B$.

It is typical in this context to represent a real number $x\in\R$ in one of two ways:

\begin{enumerate}
    \item A \textbf{name}, i.e. a rational sequence $q_n$ such that $|x-q_n| \leq 2^{-n}$ for all $n$.
    \item An upper/lower enumeration of its Dedekind cut, i.e. a pair of functions $L: \N \rightarrow \Q$, $R:\N\rightarrow \Q$ such that $im(L) = \{q\in\Q: q< x\}$, $im(R) = \{q\in\Q: q > x\}$.
\end{enumerate}

These are equivalent data: when coded as elements of $2^\N$, there is a uniform algorithm which takes you from one to the other.

If either of these pieces of data is computable from $0\in 2^\N$, the real number $x$ is called computable. If $L$ is computable, $x$ is called left-c.e. or c.e. from below. If $R$ is computable, $x$ is right-c.e. or c.e from above.

Clearly, if a number is right c.e. and left c.e. it is computable.

We say that a function $f(x)$ is computable (or c.e.) in $x$ if there is a single, uniform Turing functional which outputs $f(x)$ (or enumerates it) when $x$ is written on the read-only tape.

As mentioned, there are analogous ways to encode the data of a measure. Either by taking a function $q_n(\sigma)$ from cylinders to names, or a pair of functions $L_\sigma$, $R_\sigma$ enumerating the value on $\sigma$ from below or above.

If we pursue the $L$ and $R$ route for measures specifically, we can actually throw $R$ away entirely, as enumerating $f(\sigma)$ from above is equivalent to enumerating $\sum_{|\tau| = |\sigma|, \tau \neq \sigma} f(\tau)$ from below.

The way we will be presenting measures, therefore, is by enumerations $f_\mu: \N \rightarrow2^{<\N}\times \Q$, where $im(f) = \{(\sigma,q)\in 2^{<\N}\times\Q : \mu(\sigma) > q\}$.

A Turing functional takes in information about its input $A$ which is both positive and negative, i.e. $A(n) = 1$ or $A(n) = 0$. We are presenting our measure $\mu(\sigma)$ using $L$, i.e. using only the positive information about if the value of $\mu$ exceeds $q$ on $\sigma$. Although this information is enough to completely determine $\mu$, and satisfies the metamathematical concern of being able to ascertain values on cylinder sets, it also limits the resources available to us when we compute things based on the data of $\mu$.

\begin{dfn}
    An \textbf{enumeration reduction} of a set $B$ to a set $A$ is a computable function $\Phi : \N \rightarrow Fin(\N) \times  \N$ such that $B = \{ \Phi_2(n) :\Phi_1(n)\subseteq A, n\in\N \}$
\end{dfn}

Intuitively this means that there is a language of proof with a computably enumerable list of axioms of the form $(n_1 \in A)\ \land (n_2 \in A)...\land(n_k\in A) \implies s\in B$, such that this accurately describes the membership of $B$.

This relation $B\leq_e A$ is a preorder on $2^{<\N}$, and when the quotient to make it a partial order is taken, the resulting structure is called the \textbf{enumeration degrees}.

Like the Turing functionals, the possible enumeration reductions are numbered, but unlike the Turing functionals there is a universal such reduction, i.e. there exists $\Phi:  \N\times \N \rightarrow Fin(\N) \times  \N$ such that $\Phi(*,n)$ is the $n$-th set enumeration reducible to the input set $A$. 

We will be generating our alternative hypotheses from our null hypotheses by using an enumeration reduction from the (undergraph of the) alternative to the (undergraph of the) null. The existence of the universal reduction means that given a measure $\nu$, we may produce a semimeasure $\mu\leq_e \nu$ such that for any semimeasure $\rho \leq_e \nu$, $\exists \epsilon >0$ $\mu(\sigma) \geq \epsilon\rho(\sigma)$ for all $\sigma \in 2^{<\N}$. We call this $\mu$ the universal semimeasure enumerable from $\nu$. We will then use the following characterization of randomness as our definition. 

\begin{dfn}
      A real $x$ is  random for a measure $\nu$ if and only if $\liminf_n \frac{\nu([x|_n])}{\mu([x|_n])} > 0$ for every $\mu \leq_e \nu$, or equivalently the universal such $\mu$.
\end{dfn}

\subsection{The Orders on Measures}

Presenting measures in this way gives easy access to another helpful tool: (continuous) semimeasures. For our purposes, semimeasures are given by functions $f:2^{\N} \rightarrow [0,1]$ just like measures, and follow all the rules for measures save that rather than being additive, they are super-additive. $f(\sigma \space\hat{ {} }\space 1) + f(\sigma \space\hat{ {} }\space 0) \leq f(\sigma) $. Our scheme for the description of semimeasures will be identical to that for describing measures.

Another, possibly more fruitful, way to think about them is as the lower bound for a class of measures. $i.e.$ a semimeasure is equal to a pointwise minimum of two measures. Indeed it is a good idea to remember that positive measures (and semimeasures) are ordered by majorization.

We can introduce a secondary preorder on measures as well, which we call projective domination:

\begin{dfn}
    A (semi-)measure $\mu$ \textbf{projectively dominates} a (semi-)measure $\nu$ (written $<_p$) if $\mu(\sigma) > q \nu (\sigma)$ for some $q > 0 $ and all $\sigma$. 
\end{dfn}

Projective domination is a strong form of absolute continuity of measures. It also goes right to the heart of the question of pseudorandomness.

\begin{example}
    Suppose we have two probability measures $\mu$ and $\nu$ on $2^{\N}$ such that $\mu >_p \nu$. $\mu$ decomposes into the sum of positive measures $(\mu - \lambda\nu) + \lambda \nu$. In statistical terms, this means that $\mu$ can be obtained by
    \begin{enumerate}
        \item Flipping a coin with weight $\lambda$.
        \item If heads, selecting a real randomly using $\nu$
        \item If tails, selecting a real randomly using $\frac{1}{1-\lambda}(\mu - \lambda\nu)$
    \end{enumerate}

    As such, one will never be able to tell with confidence $>\lambda$ that a real was selected using $\nu$ rather than $\mu$.
\end{example}

If $\mu <_p \nu$ and $\nu <_p \mu$, a relation we will write as $\mu =_p \nu$, or ``$\mu$ and $\nu$ are projectively equivalent," then there is a bound on the statistical confidence one can have in discriminating between the two by looking at randomly generated reals. 

On the other hand, suppose $\mu \not{>_p} \nu$. Then for each $r > 0$, the set $A_r = \{\sigma\in 2^{<\N} : \mu(\sigma) < r \nu(\sigma)\}$ is nonempty for each $r$, and so $A = \cap_{r >0} A_r$ is nonempty by compactness.

Each element of $A$ is a real which is much more likely to have been generated by $\nu$ than by $\mu$, as in with those outcomes any statistical test of any confidence level would rule out $\mu$ in favor of $\nu$ as an alternate hypothesis.

\subsection{Randomness and Complexity Theory}

The basic definition we want to use for a test is as follows:
\begin{dfn}
    A c.e. test $T$ is an enumeration reduction of the undergraph of an output semimeasure to the undergraph of an input measure. In other words a c.e. test $T$ is a c.e. collection of elements in $(2^{<\N} \times\Q )^{<\N}\times 2^{<\N} \times\Q $ such that when $\mu$ is a measure, 
    $$\left\{ (\sigma,q): \left(\prod_{i=0}^k(\sigma_i,q_i)\right)\times (\sigma,q)\in T, \space \land_{i=0}^k \mu(\sigma_i) \geq q_i\right\}$$ is the undergraph of a semimeasure.
\end{dfn}

We will give a more general definition later, which allows us to drop the c.e. condition. The tests we consider are the same as those in \cite{daymiller}, but framed in terms of semimeasures rather than $L^1$ functions enumerable from below.

If we drop the requirement that $T$ itself be a c.e. set, and just maintain that it must enumerate the undergraph of a semimeasure, we will still call $T$ a test. 

We say that a measure $\mu$ weakly passes the test $T$ if $\mu \geq_p T(\mu)$ and passes the test if $\mu \geq T(\mu)$ pointwise.

In this setting, two questions immediately present themselves:

\begin{enumerate}
    \item Does every test have a measure which passes it?
    \item What properties can we bake into a test so that measures which pass it have interesting or desirable properties themselves?
\end{enumerate}

The answer to the first question turns out to be yes: as a consequence of Sperner's lemma, Brouwer's theorem, or Kakutani's theorem, arbitrary tests will have measures which pass for them. The original proof was by Levin in \cite{levin}, but the details were largely omitted. The Sperner's lemma proof is done in full by G\'acs in his lecture notes \cite{gacslecturenotes}. The Kakutani proof was done by Day and Miller in \cite{daymiller}. We recreate it here in section \ref{sec:Neutral} using semimeasures, which somewhat simplifies the proof.
Superadditivity for the maximum value of finite strings on $2^{<\N}\times\Q$ is a computable condition, so the universal enumeration reduction can be made into a test which projectively dominates every other test on any input as follows:
\begin{itemize}
    \item Enumerate the attempts to enumerate semimeasures from the undergraph of a measure
    \item When any attempt at including a new element results in the superadditivity condition being violated, do not include that element
    \item Sum the attempts with weight $2^{-n}$ to produce a semimeasure
\end{itemize}

Any semimeasure enumeration reducible to the input measure will appear on the list somewhere, so the output semimeasure will projectively dominate it. We call a test generated in this way (or any projectively equivalent test) a universal c.e. test.

\begin{dfn}
    A measure which passes a universal c.e. test is called neutral.
\end{dfn}

If a measure is neutral, it dominates the (projective order) maximal semimeasure in its enumeration degree, which means that it is a maximal semimeasure in its enumeration degree. 
This makes the answer to the second question all the more interesting. If there is a property of a measure that can be expressed in terms of passing a test, there is some measure which satisfies that property. The properties which can be encoded in tests, which are discussed in more detail in \ref{sec:App}, tend to be positive requirements of the support of a test. The a priori clear example of such a property is that mass should be concentrated in such a way as to prevent any reals from being declared non-random. A strong version of this property is that the output of a (specific choice of) universal c.e. test should be dominated by the input measure, meaning that our best alternative hypothesis should be less likely than the null for every outcome of the experiment. This is the property encoded by the universal c.e. test we select, and so neutral measures have no non-random reals. This property prevents neutral measures from having Turing degree.

\begin{prop}
    If $\mu$ is a neutral measure on $2^\N$, $\mu$ does not have a unique minimal Turing degree representation.
\end{prop}
The proof is an adapted form of a result by Reimann and Slaman \cite{Reimann1}.
\begin{proof}
    Suppose it did: that there is some real $A$ which computes the cylinder set values of $\mu$ and for which any other real which may do so can compute it. We now construct a real by finite extension:
    \begin{enumerate}
        \item At step $0$, set $\sigma_0 =\emptyset$
        \item At step $s \geq 1$, for each $k$ check all strings extending $\sigma_{s-1}$ of length $\leq k$ up to step $k$ in the representation of $\mu$ by $A$
        \item Stop the search at the step $k$ when there appears an extension $\tau$ of $\sigma_{s-1}$ of length at least $s$ and $\mu(\tau) <2^{-s}$.
        \item Select the lexicographically least such $\tau$ and set $\sigma_s = \tau$
    \end{enumerate}
    $\cap_s\sigma_s$ consists of one real, which we call $\sigma$. The function $\mu_A:E\rightarrow \int_E \sum_{s}\chi_{\sigma_s}d\mu$ is a measure on $2^{\N}$ which is computable in $A$, meaning $\mu_A <_e A\oplus A^c <_e \mu$, which means there is some $\lambda$ such that $\lambda \mu_A < \mu$, but $\mu|_{\sigma_n} < \frac{1}{n}\mu_A|_{\sigma_n}$, which is a contradiction.
\end{proof}
\subsection{Random Sequences of Integers}\label{sec:Integers}
We have entirely restricted ourselves to $2^\N$ as our underlying space. This is for good reason: trying to extend to non-compact spaces results in disaster. Take for example Baire space $\N^\N$. Suppose we have a probability measure $\nu$ on this space. Define 
$$R_n = \bigcup\left\{[\sigma\hat{\space}k]:\sum_{i < k}\nu([\sigma_{}\hat{\space}i])\geq \left(1-\frac{1}{(n+2)!}\right)\nu(\sigma) \text{ and } |\sigma| = n\right\}$$
Furthermore, pick a computable enumeration $g:\N\rightarrow \N^{<\N}$ and let $f$ be the function obtained by appending trailing zeroes to the output of $g$. We then take 
$$ \mu =\left(\frac{1}{2}\right)\sum_{n\in\N} \chi_{R_n} \mu + 2^{-n-1}\delta_{f(n)}$$

Each $R_n$ has measure at most $\frac{1}{(n+2)!}$ so the above semimeasure (actually a measure in this case) has total measure $<1$. If every atom of $\mu$ is random with respect to $\nu$, then $\nu(\sigma) > 0$ for every $\sigma$. Furthermore, $\cup_{k \geq n} R_k$ is open and dense, so $\cap_n \cup _{k\geq n} R_k$ is dense and thus nonempty. If $x$ is in this set, then $\mu (x|_n) \geq n \nu(X|_n)$, and thus $x$ is non-random. By the definition of $\mu$ it is enumeration reducible to $\nu$, and thus $\nu$ has non-random reals.

Thus, we are very quickly defeated if we try to construct a neutral measure on Baire space. Any proof of the existence of neutral measures for $\mu$ will rely on the compactness of the space (or at least that it has a compact exhaustion). This is in stark contrast to the non-computable measure theory of $\N^\N$, as Baire space is, in fact, Borel equivalent to Cantor space.  

\section{The Space of Randomness Tests}\label{sec:Tests}
\subsection{Defining Randomness Tests}\label{sec:TestDef}

We define our representations of measures and semimeasures for the sake of convenience.

\begin{dfn}
    A Borel probability measure on $2^\N$ is a countably additive function $\mu$ from the $\sigma$-algebra of Borel subsets of $2^\N$ to $[0,1]$ such that $\mu(2^n)=1$. We identify a measure with its values on cylinder sets, and thus identify the measures with the functions $\mu: 2^{< \N}\rightarrow [0,1]$ such that $\mu(\emptyset) = 1$ and $\mu(\sigma  ) = \mu(\sigma\space\hat\space 1)+\mu(\sigma\space\hat\space 0)$.
\end{dfn}

\begin{dfn}
    The space of continuous probability semimeasures $SM(2^\N)$ is the space of functions $\mu: 2^{< \N}\rightarrow [0,1]$ such that $\mu(\emptyset) = 1$ and $\mu(\sigma  ) \geq \mu(\sigma\space\hat\space 1)+\mu(\sigma\space\hat\space 0)$.
\end{dfn}
Every measure is a semimeasure. Furthermore, we may order semimeasures by pointwise domination on cylinder sets, which will be written $\mu < \nu$. In this order, measures are maximal elements, and each semimeasure may be identified its upper cone, which is a closed convex set. We endow both measures and semimeasures with the subspace-product topology on $[0,1]^{2^{<\N}}$.

Classically, a randomness test for a real in $2^\N$ and some probability measure $\mu$ on $2^\N$ is a countable set of increasingly unlikely properties such that the intersection of which has zero $\mu$-measure. A real will fail to be random if it satisfies all of these properties. This is the notion of randomness captured by Martin-L\"{o}f (see Downey and Hirschfeldt \cite{DH} for details), and is sufficient for computable measures. Usability concerns require us to limit ourselves to finitary properties (cylinder set membership), and further that they be computable sets, so computable open sets. This gives that ultimately randomness should be membership or otherwise in a computable null $G_\delta$ set.
More generally, subsequent researchers such as Levin \cite{levin}, Reimann and Slaman \cite{Reimann1}, and Day and Miller \cite{daymiller}, have explored the question for non-computable measures. The main difficulty of this generalization is a necessary ambiguity in the representation of a measure. We would still like whatever randomness test we use to be computable \textit{relative} to some representation of the measure, but the choice of representation changes the computational resources in a nontrivial way. Without a change in methodology, one is forced to call a real random if there is at least some representation in which it is so, as there will not be any real which is random for every possible representation.
This ambiguity is, of course, absent when the measure has a computable representation in the first place, as one can pick an arbitrary computable representation, and any other such choice will be equivalent. In the case where the measure $\mu$ has Turing degree (meaning that there is a Turing degree of a representation such that every representation computes that Turing degree), there is a similar lack of ambiguity.

G\'{a}cs \cite{Gacs}, Levin \cite{levin}, and Day and Miller \cite{daymiller} use the notion of a \textit{uniform test}, which captures the same randomness notion. This notion is our starting point, but we extend it considerably, first by loosening the computability restrictions and second by reconsidering the notion in the language of continuous semimeasures. The first component we will need for our notion of test follows.

\begin{dfn}
    The space $A$ of conditions on measures is the subspace of the space of finite-domain partial functions $$f:2^{<\N} \rightarrow \left\{\frac{n}{2^k}: n\in\N, k\in \N\right\}^2 $$
    such that there is a measure $\mu$ with $f_1(\sigma) < \mu(\sigma) < f_2(\sigma)$ for all $\sigma\in dom(f)$. We say that such a $\mu$ satisfies the condition $f$ or write $\mu\models f$

\end{dfn}

We note that this second condition is computable as we only need to take measures supported on finitely many atoms of the form $\sigma\hat{\space}(0,0,0...)$. We may identify a measure $\mu$ with the set of conditions it satisfies.

The set $A$ of all conditions on measures captures the notion of \textit{finitary data} about a measure. This includes restrictions on the breadth of the information (that the domain of $f$) be finite, and its depth (that the number of bits that $f$ tells us about a measure satisfying it) is finite. If we were to think of $2^\N$ as $2^{\Z^n}$, an arrangement of bits in space, a condition would be asking about the probabilities of some finite collection of local patterns (cylinder sets) up to a finite precision. This is a more physical interpretation of the same data, which will come up in more depth when we discuss applications.

In the language of hypothesis testing, we want our tests, based on some knowledge (that which is representable in A) of the measure which serves as our null hypothesis, to generate an alternative hypothesis. An alternative hypothesis ought to be another measure by which a real might have been generated, but we will allow ourselves to encode convex classes of measures by taking their pointwise minimum on cylinder sets. This produces a continuous semimeasure.

A continuous semimeasure (henceforth we will drop the term continuous and just say semimeasure), will only be a more credible hypothesis than a measure itself if \textit{every element} of the class of measures it encodes is a more credible alternative hypothesis. We collect the set of data about a semimeasure we may consider in the definition that follows.

\begin{dfn}
    The space $B$ of conditions on semimeasures is the space of semimeasures $\mu$ with rational values such that there exists $n$ with $|\sigma| > n \implies \mu(\sigma) = 0$. A semimeasure $\nu$ satisfies $\mu$ if $\nu \geq \mu$.
\end{dfn}

A semimeasure is the (pointwise) supremum of the conditions it satisfies, and thus may also be identified with this set. Hence, an enumeration of a semimeasure from below is the same as an enumeration of this set. Ultimately, our notion of test will involve obtaining a semimeasure from a measure, and these collections of data will allow us to build a framework for doing so. An important idea for this framework is that both of these collections carry natural partial orders as follows:
$$ f \leq_A g \Longleftrightarrow (dom(f) \subseteq dom(g))\land \forall \sigma\in dom(f)([f_1(\sigma), f_2(\sigma)]\supseteq [g_1(\sigma),g_2(\sigma)]) $$

This order is equivalent to reverse inclusion on the set of measures which satisfy the conditions in question.

$$x\leq_By \Longleftrightarrow\forall\sigma (x(\sigma)\leq y(\sigma))$$
Likewise, this order is equivalent to reverse inclusion on the set of semimeasures satisfying both conditions.

\begin{dfn}
    A \textit{test} on $2^\N$ is nondecreasing function $T:A\rightarrow B$
\end{dfn}

This means that a test is simply a set of inferences from $A$ to $B$ which is compatible with the strength of the associated conditions, i.e. learning more information about the class of measures cannot result in knowing less about the class of semimeasures to which it corresponds.

A subset $S$ of $A$ will be called pairwise compatible if for any $p,q\in S$, there is $r$ with $p,q <_A r$ and downward closed if $p\in S$ and $q <_A p$ imply $q\in S$. A measure on $2^\N$ is equivalent to the set of conditions in $A$ it satisfies, which is a maximal pairwise compatible and downward closed set. If we identify the two and set
$T(S) = \sup_{s\in S} T(s)$ in an abuse of notation, we obtain a map $T$ from measures to semimeasures which is simply $T(\mu) = \sup_{\mu\models a}T(a)$

\begin{dfn}
    A measure $\mu$ passes a test $T$ if $T(\mu) \leq \mu$ pointwise. It passes a test weakly if there is some nonzero $q\in \Q^+$ such that $\mu$ passes $qT$.
\end{dfn}

In the language of hypothesis testing, passing a test $T$ means that $\mu$ is at least as credible a hypothesis for generating any real as some alternative in the family $T(\mu)$. This makes measures which pass sophisticated tests somewhat perverse and pathological: it is impossible to rule them out as null hypothesis using $T$, and if $T$ is sophisticated, it is therefore difficult to rule them out at all. Weakly passing a test means that it is at least $1/q$ times as likely to generate any real using some alternative in $T(\mu)$.

We recreate the proof of Day and Miller \cite{daymiller} for our notion of test (c.e. and otherwise), for the sake of completeness, in section \ref{sec:Neutral}. This will allow us to prove general existence theorems for different classes of measure, as well as enforce strange properties of neutral measures, which will be done in section \ref{sec:App}.

\subsection{Universal Tests}\label{sec:Univ}
\begin{dfn}
    If $S$ and $T$ are tests, we say $S\geq T$ if $S(\mu)[\sigma]\geq T(\mu)[\sigma]$ for all $\mu$ and $\sigma$. We say $S\geq_p T$ if there exists $q\in \Q^+$ such that $S\geq qT$.
\end{dfn}

\begin{dfn}
    Let $\Omega$ be a class of tests. We call a test $T$ $\Omega$-hard if for every $\omega\in\Omega$ $T\geq_p \Omega$. It is called $\Omega$-universal if it is $\Omega$-hard and an element of $\Omega$.
\end{dfn}

\begin{dfn}
    A measure $\mu$ is neutral for the class $\Omega$ if it passes an $\Omega$-hard test. If $\Omega$ is the class of c.e. tests, $\mu$ is just called neutral.
\end{dfn}

We would like to address the question of which classes of tests have hard and universal tests. In order to do so, we must first establish the analogous
facts about semimeasures (viewing them as constant tests). We begin by defining an auxilliary function $M$, which gives us quantitative information about if and by how much a set of semimeasures fails to be dominated by a particular probability measure. We then use this function to establish necessary and sufficient conditions for a class of semimeasures to be projectively dominated by a single probability measure. We then turn our attention to the analogous questions about tests.\\
\begin{dfn}
    A \textit{non-normalized} semimeasure on $2^{\N}$ is a superadditive function $2^{<\N}\rightarrow [0,\infty]$. We endow this with the subset product topology. The space of such functions is denoted $SM^+(2^\N)$. 
\end{dfn}
$SM^+(2^\N)$ is ordered by pointwise domination and projectively. $SM(2^\N)$ includes continuously in $SM^+(2^\N)$ and the inclusion is order preserving. Suppose $X\subset SM^+(2^\N)$ is a closed set of superadditive functions $2^{<\N}\rightarrow [0,\infty]$.
\begin{lem}\label{thm:fundSM}

    The mass function 
    $$M(X,\sigma) = \sup_{R\subset 2^{<\N} \text{ prefix free} }\sum_{\rho\in R }\max_{\mu\in X}\mu (\sigma \hat{\space}\rho)$$
    is a non-normalized semimeasure as a function of $\sigma$ for fixed $X$, and in particular is the least non-normalized semimeasure greater than each element of $X$.
\end{lem}
\begin{proof}
${}$\\
\begin{enumerate}
    \item That $M(X,\sigma) \geq \mu$ for any $\mu\in X$ follows directly from taking $R$ to be the empty string as a witness.
    \item Suppose $M(X,\sigma) > g(\sigma)$ where $g$ is a semimeasure above every element of $X$. There is some witnessing $R'$ with
    $$ \sum_{\rho\in R'}\max_{\mu \in X}\mu(\sigma\hat{\space} \rho) > g(\sigma)$$
    we may replace $R'$ with a finite subset $R$ while keeping the inequality, and thus assign $\mu_\rho$ to be the maximizing $\mu\in X$ for that $\rho\in R$. We then have
    $$ g(\sigma) \geq\sum_{\rho \in R} g(\sigma\hat{\space} \rho) \geq \sum_{\rho \in R} \mu_\rho(\sigma\hat{\space} \rho) > g(\sigma) $$
    since $g$ is a semimeasure, which is a contradiction. Thus $M(X,\sigma)$ is below any semimeasure pointwise in $\sigma$ which bounds $X$ from above.
    \item If the supremum over prefix free subsets $R\subset 2^{<\N}$ is restricted to prefix free subsets not containing the empty string, then $R$ can be divided into the parts $R_0$ and $R_1$ which begin with $0$ and $1$, and thus
    $$ M(X,\sigma) \geq \sum_{i =0,1} \sup_{R_i\subset 2^{<\N} \text{ prefix free} }\sum_{\rho\in R_i }\max_{\mu\in X}\mu (\sigma \hat{\space}i\hat{\space}\rho) = \sum_{i = 0,1} M(X,\sigma \hat{\space} i)$$
\end{enumerate}
\end{proof}

\begin{lem}\label{thm:MeasureExists}
    There is a probability measure dominating $X$ iff $M(X,\emptyset) \leq 1$
\end{lem} 
\begin{proof}$${}$$
($\impliedby$)\\
The measure is defined by induction on  $|\sigma|$
\begin{enumerate}
    \item $\mu(\emptyset) =1$
    \item Now assume $\mu$ is defined on strings of length $n$ and $\mu(\sigma)\geq M(X,\sigma)$ for all such strings.
    \item For each string $\sigma$ of length $n$, assign $\mu(\sigma \hat{\space}0)\geq M(X,\sigma \hat{\space}0)$ and $\mu(\sigma\hat{\space}1) \geq M(X,\sigma \hat{\space}1)$ so they add up to $\mu(\sigma)$.
    \item Repeat the process.
\end{enumerate}
($\implies$)\\
$M(X,\sigma) \leq \mu(\sigma)$ for all $\sigma$ by Lemma \ref{thm:fundSM}. Therefore $M(X,\emptyset) \leq \mu(\emptyset) = 1$
\end{proof}

\begin{lem}\label{pseudolinearity}
    For $\lambda >0$, let $\lambda X$ denote $\{\lambda\mu: \mu\in X\}$. 
    $$M(\lambda X,\sigma) = \lambda M(X,\sigma)$$
    $$M(X \cup Y,\sigma )\leq M(X,\sigma ) + M(Y,\sigma) $$
\end{lem}
\begin{proof}
    In the first case, we may pass $\lambda$ through each step:
    \[M(\lambda X,\sigma) = \sup_{R\subset 2^{<\N} \text{ prefix free} }\sum_{\rho\in R }\max_{\mu\in X}\lambda\mu (\sigma \hat{\space}\rho) = \lambda M(X,\sigma) \]
    For the second, observe that 
    \[\max_{\mu\in X\cup Y}\mu(\sigma\hat{\space}\rho) \leq \max_{\mu\in X}\mu(\sigma\hat{\space}\rho)+\max_{\nu\in Y}\nu(\sigma\hat{\space}\rho)\]
    and making this substitution termwise yields
    \[\sup_{R\subset 2^{<\N} \text{ p.f.} }\sum_{\rho\in R }\max_{\mu\in X\cup Y}\mu(\sigma\hat{\space}\rho) \leq  \sup_{R\subset 2^{<\N} \text{ p.f} }\sum_{\rho\in R }\left(\max_{\mu\in X}\mu(\sigma\hat{\space}\rho)+\max_{\nu\in Y}\nu(\sigma\hat{\space}\rho)\right)\]
    which by definition gives
    \[M(X \cup Y,\sigma )\leq M(X,\sigma ) + M(Y,\sigma) \]
\end{proof}

\begin{dfn}
    We call a subset $X\subset SM^+(2^\N)$ (not necessarily closed) \textit{finite-mass} if its closure $\overline{X}$ has the property that $M(\overline{X}, \emptyset) < \infty$. We call such a set $\sigma$-\textit{finite-mass} if it is a countable union of finite-mass sets.
\end{dfn}

\begin{thm} \label{thm:sigmaFinite}
    Fix a subset $X\subset SM(2^\N)$, there exists a measure $\mu$ such that $\forall \chi\in X \exists t > 0 (\mu \geq t\chi)$ iff $X$ is $\sigma$-finite-mass.
\end{thm}
\begin{proof}
    ${}$\\
    ($\implies$)\\
    For each $t\in \Q^+$ set $X_t = \{\chi\in X: t\chi \leq \mu\}$. $M(t\space\overline{X_t}, \sigma) \leq \mu(\sigma)$, so $X_t$ is finite-mass. $X = \cup_{t\in\Q^+} X_t$, so $X$ is $\sigma$-finite-mass.\\
    ($\impliedby$)\\
    Suppose $X$ is $\sigma$-finite-mass, and let $X_n$ be a sequence of finite mass subsets witnessing this fact. Take $\mu_n$ a probability measure such that 
    $$\forall\sigma \left(\mu_n(\sigma)\geq\dfrac{M(\overline{X_n},\sigma)}{M(\overline{X_n},\emptyset)}\right)$$. Take $\mu = \sum_{n\in\N} 2^{-n-1}\mu_n$. $\mu$ is a probability measure, and for $\chi\in X$, there is some $n$ with $\chi\in X_n$, and so $\dfrac{2^{-1-n}}{M(\overline{X_n},\emptyset)}\chi \leq \mu$.
\end{proof}

At minimum, for a class of tests to have a corresponding hard or universal test, the collection of semimeasure outputs of these tests on an individual input measure must be $\sigma$-finite-mass. Otherwise, that individual input would not result in a valid probability semi-measure output.

 We need another helpful definition:

\begin{dfn}
    A test $T$ will be said to have depth $n$ if it satisfies all of the following conditions:
    \begin{enumerate}
        \item $\forall\mu\in Pr(2^{\N})(|\sigma| > n\implies T(\mu)[\sigma] = 0)$
        \item For any $f\in A$, $\exists g \leq_A f$ with $dom(g) \subset 2^n$ and $T(g) = T(f)$
    \end{enumerate}
\end{dfn}

\begin{lem} \label{thm:finitedepth}
    A test $T$ is recoverable as the supremum of a set of finite depth tests. 
\end{lem}
\begin{proof}
    Define $R_n:A\rightarrow A$ on $f\in A$ by restricting the domain to $2^n$. $R_n(f)$ is thus a condition satisfying the constraints on $g$ in point $2$ of the definition of finite depth. We have that $R_n$ nonstrictly preserves the order on $A$ and $sup_n R_n(f) = f$.
    Furthermore, define $P_n:B\rightarrow B$ by setting $P_n(\mu)[\sigma] = 0 $ for $|\sigma | > n$ and leaving it as $\mu(\sigma)$ otherwise. $P_n$ is an order preserving (again not strictly) idempotent operator, and $\sup_n P_n(\mu) = \mu$. Set $T_n = P_n\circ T\circ R_n$. $T_n$ is a test of depth $n$. For $a\in A$, $b\in B$, there is an $n$ for which $R_k(a) = a$ and $P_k(b) = b$ for any $k > n$ (namely this $n$ will be any upper bound on the length for $supp(b)$ and $dom(a)$). We thus have $sup_n T_n = T$ because the two will be eventually equal for each individual point.
\end{proof}

It is also worth noting that $T_n\leq T_{n+1}$. 

\begin{dfn}
    Let $\Omega$ be a closed set of tests, and let $\Omega_n = \{P_n\circ T\circ R_n: T\in \Omega\}$. Define $M_\Omega (\mu,\sigma) = \sup_n M(\{T(\mu), T\in \Omega_n\},\sigma)$. $M_\Omega$ is a non-normalized semimeasure valued function on measures.
\end{dfn}

\begin{lem}\label{Thm:FiniteMassTest}
    If there exists a test $T$ which dominates every element of $\Omega$, then $M_\Omega \leq T$ pointwise on measures and cylinders and $M_\Omega$ dominates every element of $\Omega$.
\end{lem}
\begin{proof}
    $M_\Omega$ dominates every element of $\Omega_n$ by definition. For $S\in \Omega$
    $$S(\mu)[\sigma] = \sup_n P_n\circ S\circ R_n(\mu)[\sigma] \leq\sup_n M(\mu,\Omega_n,\sigma) =M_\Omega(\mu,\sigma)$$.\\
    If there is a $\mu$ and $\sigma$ for which $T(\mu)[\sigma] < M_\Omega(\mu,\sigma)$, there is some $\Omega_n$ for which $M(\{S(\mu):S\in \Omega_n\},\sigma) > T(\mu)[\sigma]$, which contradicts that $T$ dominates each element of $\Omega$.
\end{proof}

\begin{lem}\label{normimpliestest}
    If $M_\Omega$ is normalized for every input, there is a test $T$ which is equal to $M_\Omega$ on any measure.
\end{lem}
\begin{proof}
    Let $f_k$ be the function which takes a semimeasure $m(\sigma)$ to $m'(\sigma) =2^{-k}\mathrm{floor} (2^km(\sigma))$, which is also a semimeasure. 
    $$M_\Omega^n:= M_{\{f_n\circ  P_n\circ T\circ R_n:T\in \Omega\}}$$ is thus a test. We have
    $$\sup_nM_\Omega^n = M_\Omega$$
    and we define a test $\tilde{M}_\Omega$ by taking $a\in A$ with $n$ the least number such that $Dom(a) \subseteq 2^n$ to $M_{\Omega}^n(a)$. $\sup_{\mu\models a} \tilde{M}_\Omega(a)[\sigma] = M_{\Omega}(\mu,\sigma)$
\end{proof}
We will thus regard $M_{\Omega}$ as identified with this test without loss of generality. $M_\Omega$ is the least test above every element of $\Omega$

\begin{lem}\label{psuedo2electricboogaloo}
    $tM_{\Omega} = M_{t\Omega}$ and $M_{\Omega_1\cup\Omega_2} \leq M_{\Omega_1} + M_{\Omega_2}$
\end{lem}
\begin{proof}
    This is the same as \ref{pseudolinearity}
\end{proof}
\begin{thm}\label{thm:HardTests}
    There exists a test dominating a set of semimeasures $\Omega$ iff $M_\Omega$ is finite. There exists a hard test for $\Omega$ iff we may write $\Omega$ as a countable union $\Omega = \cup_k \Omega_k$ such that $\sup_{\mu}M_{\Omega_k}(\mu,\emptyset) < \infty$
\end{thm}
\begin{proof}
    The first claim is the consequence of \ref{normimpliestest} and \ref{Thm:FiniteMassTest}. If $\Omega$ may be written as such a union, we may take $$\Omega' = \bigcup_k \frac{2^{-k}}{\sup_{\mu}M_{\Omega_k}(\mu, \emptyset)}\Omega_k$$
    $\Omega'$ is dominable by \ref{psuedo2electricboogaloo}, and any $\Omega'$-hard test is also $\Omega$-hard. Likewise, if there exists an $\Omega$-hard test $T$, for $t\in\Q^+$ we can take $\Omega_t =\{\omega\in \Omega: T \leq t\omega\}$, and $\cup_t\Omega_t = \Omega$, and $\sup_{\mu}M_{\Omega_t}(\mu,\emptyset) < t^{-1}$.
 \end{proof}

\begin{cor}
    Any countable family $\Omega$ has a hard test.
\end{cor}

The question of universality is more subtle. In particular, there are classes which have hard tests but not universal ones.

\begin{example}
    Let $r = (0,0,...)$, and define $\Omega$ to be the set $\{\delta_{\sigma\hat{\space} r}: \sigma \in 2^{<\N}\}$. There are many $\Omega$-hard tests (just pick a summable positive sequence), but none of those tests are universal. Furthermore, there is no projectively smallest $\Omega$-hard test. Given a hard semimeasure $\mu$ enumerate $\Omega$ as $\omega_n$ and assign to each $n\in \N$ the minimum real $m_{n}$ for which $m_n\mu \geq \omega_n$ pointwise. WLOG assume that $\mu$ is supported entirely on the supports of $\Omega$. The semimeasure $\sum_{n\in\N} 2^{-n}m_n \omega_n$ is strictly projectively smaller than $\mu$, but is still $\Omega$-hard.
\end{example}

\section{The Fixed Point Problem}\label{sec:Neutral}

We now begin proving the existence of neutral measures. As alluded to, we will (as did \cite{daymiller} and \cite{levin}) do this by characterizing neutral measures as those which satisfy a certain fixed point property. We apply the following, proved originally in \cite{kakutani}:
\begin{thm}[Kakutani Fixed Point Theorem]
    Let $K\subset \R^n$ be convex and compact, and let $f:K\rightarrow 2^K$ satisfy
    \begin{itemize}
        \item $\Gamma f =\{(x,y): y\in f(x)\}$ is closed
        \item $f(x)$ is convex and closed for every $x$
    \end{itemize}
    then there exists at least one $x$ such that $x\in f(x)$. The set of such points is obtained by projecting the intersection of the diagonal with $\Gamma f$, and is therefore a nonempty compact subset of $K$.
\end{thm}
As with most fixed point theorems of its kind, Kakutani relies on finite dimensionality for its application. The lemma \ref{thm:finitedepth} tells us that we can restrict ourselves to considering finite depth tests. This is because while finite depth tests (like all tests) produce semimeasures from measures on $2^\N$, they factor through maps on finite dimensional convex sets, and thus we can apply Kakutani's theorem as follows:

\begin{lem}\label{Thm:KakFiniteDepth}
    The set of probability measures passing a finite depth test is compact and nonempty.
\end{lem}
\begin{proof}
    Let $T$ be a test of depth $n$. $T(\mu)$ depends only on the restriction of $\mu$ to the cylinders of length $n$. Furthermore, its only nonzero outputs are on these same cylinders and initial segments of the same. $T$ thus essentially takes inputs in the simplex
    
    $$\Delta =\left\{x\in[0,1]^{2^n}:\sum_{i\in 2^n} x_i = 1\right\}$$
    The outputs are semimeasures which place no restrictions on cylinders of length greater than $n$, which can be identified with subsets of $\Delta$ of the form
    $$S_r= \left\{x\in\Delta:\forall\sigma, |\sigma| \leq n, \sum_{i \text{ extending } \sigma}x_i \geq r(\sigma)\right\}$$
    for a given semimeasure $r$, which is a class of nonempty closed convex subsets of $\Delta$. If we identify the output of $T$ with its values on $2^n$, and allow $T$ to descend to a well defined function $t$ from $\Delta$ to semimeasures on $2^n$. Thus our object is an $x\in \Delta$ such that $x\in S_{t(x)}$. In order to use Kakutani's theorem, we need to establish that the graph 
    $$ \bigcup_{x\in\Delta} \{x\}\times S_{t(x)}$$
    is closed as a subset of $\Delta^2$. To do this, we must establish that for any $(x,y)$ outside of this graph, there is an open neighborhood of it disjoint from the graph as well. If $y(\sigma) < t(x)[\sigma]$, then there is some condition $a\in A$ which (any extension to $2^\N$ of) $x$ satisfies, and which is mapped under $T$ to a condition $z$ in $B$ with $y(\sigma) < z(\sigma) \leq t(x)[\sigma]$. Since $a$ defines an open neighborhood $U_a$ of $x$, and the set $ \Delta\setminus S_z$ is an open neighborhood of $y$, $U_a\times (\Delta\setminus S_z)$ is an open neighborhood of $(x,y)$ disjoint from the graph. Thus, the conditions of Kakutani's theorem are satisfied, and the set of points such that $x\in S_{t(x)}$ is nonempty.\\
    Furthermore, this set is the projection of the intersection of the closed graph of the test $t$ with the diagonal, so it is a compact subset of $\Delta$, and thus its preimage under projection to $2^n$ is a compact set of measures as we have chosen a compact topology on $P(2^\N)$. 
    
\end{proof}

Lemma \ref{thm:finitedepth} says that we can realize our target test $T$ as a supremum of finite depth tests. This allows us to pass from finite depth tests to any test.

\begin{thm}\label{thm:mainthm}
    Given a test $T$, there exists a measure which passes it.
\end{thm}

\begin{proof}
    Given a test $T$, form the sequence $T_n$ as in lemma \ref{thm:finitedepth}. The set of measures passing $T$ will be the intersection of those passing the $T_n$, since $T = sup_n T_n$. By the lemma \ref{Thm:KakFiniteDepth}, each $T_n$ is passed by a nonempty compact set of probability measures. These sets are nested, and therefore their intersection is nonempty. Thus there is a measure $\mu$ which passes $T$.
\end{proof}

\begin{cor}
    There exists a nonempty $\Pi_1^0$ class of neutral measures.
\end{cor}
\begin{proof}
    Let $T$ be a universal c.e. test. $T(\mu)[\sigma]$ is lower c.e. and $\mu(\sigma)$ is upper c.e. in any enumeration of $\mu$. Thus there is a finite stage at which it will be clear that $\mu(\sigma) < T(\mu)[\sigma]$ if it is true. If $\mu$ is not neutral, there is therefore a finite $\sigma$ and finite stage for which $\mu(\sigma) < T(\mu)[\sigma]$, and so the neutral measures are a $\Pi_1^0$ class. Nonemptiness is the content of the previous theorem.
\end{proof}

\section{Applications of Different Types of Test}\label{sec:App}
\subsection{Tests Which Ignore the Input}\label{sec:Ignore}
If $p$ is a c.e. measure (or semimeasure), then $P(\mu) := p$ is a c.e. test. Therefore, if $T$ is a universal test, $T >_p P$ as tests, and as any neutral measure will pass such a test, it will also projectively dominate $p$. We can use this to show that neutral measures must have a quantity of mass assigned in particular ways. The following is the easiest example:

\begin{example}
    If $x\in 2^{\N}$ is a computable real, the Dirac mass on it is a computable measure. Therefore a neutral measure will have an atom at $x$.
\end{example}

Since computable reals are dense, we may immediately deduce the following.

\begin{prop}
    A neutral measure will have a dense set of atoms.
\end{prop}

\begin{example}\label{Pq}
    Let $q\in [0,1]$ be computable. Denote by $P_q$ the Bernoulli $q$ measure, i.e. the measure produced by flipping a coin with probability $q$ of getting $1$. This measure is computable, and therefore it is projectively dominated by any neutral measure.
\end{example}
\begin{dfn}[Pointwise Dimension]
    The pointwise dimension of a real $x\in 2^\N$ with respect to a measure $\mu$ is
    $$d_\mu(x)=\liminf_{r\rightarrow 0} \frac{\log(B_r(x))}{\log(r)}$$
    this will also be referred to as $\mu$-pointwise dimension.
\end{dfn}
\begin{prop}\label{entropydensity}
    A neutral measure will have a positive probability assigned to reals of any computable pointwise dimension.
\end{prop}
\begin{proof}
    Almost every real has $P_q$-pointwise dimension $-qlog(q) - (1-q)log(1-q)$. By the ergodic theorem or law of large numbers. The function $S:[0,1/2]\rightarrow [0,1]$ defined by $S(q)= -qlog(q) - (1-q)log(1-q)$ is computable with computable inverse. For computable $d$, any neutral measure will projectively dominate $P_{S^{-1}(d)}$, and thus will assign positive probability to a set of reals each of which have pointwise dimension $d$, namely the set of such reals assigned measure one by $P_{S^{-1}(d)}$.
\end{proof}

\begin{example}
    For $\sigma \in 2^{<\N}$, let $L_\sigma$ be the uniform probability measure on the cylinder $[\sigma]$. Further, let $\Phi:2^{<\N}\rightarrow 2^{<\N}$ be the partial function defined a universal prefix free Turing machine (see \cite{DH} for details). Define a measure $\mu$ by
    \[\mu=\sum_{\sigma\in Dom(\Phi)}2^{-|\sigma|} L_{\Phi(\sigma)}\]
    This measure has total mass Chaitin's constant $\Omega < 1$ and is c.e. Thus, any neutral measure $\nu$ will projectively dominate it, so there will exist $a > 0$ such that $\nu\geq a\mu$. $\mu$ may be viewed as a measure which assigns mass to cylinders in accordance with their algorithmic simplicity, and thus $\nu$ must also assign more mass to simpler cylinders in the same way.
\end{example}
\subsection{Tests Which Spread the Input}\label{sec:Spread}
The formalism of tests as developed so far has some interesting relationships with the physics of lattice models. These are physical models of solids (or approximations to field theories), in which the underlying system is taken to be a composite of copies of a smaller system with possible states $S$ arranged in a lattice of positions depending on $\Z^n$, or more generally some other group $G$ or space carrying a transitive action thereof. These systems are assumed to interact in ways that depend only on the relative position of the lattice points. Thus, it is assumed that these are physical systems where the state is specified by a function $f:\Z^n\rightarrow S$, and the energy of a state is invariant under translations by elements of $\Z^n$ (In general by whatever group acts on the lattice; it is important only that the group be \textit{amenable}, meaning that one can take mean values across lattice sites using the tools of ergodic theory). Further work on this is forthcoming, but some of the results will be presented here as they are informative on the structure of the neutral measures themselves. It also sometimes assumed that if the state space $S$ of the individual systems carries a natural group action, that the energy of a state is invariant under the global action of that group on $f$. This is referred to as a global gauge symmetry. At some point we may wish to enforce translation invariance (on $2^G$) or global gauge invariance (on $S^{\Z^n}$ where $S$ is compact and carries an action of the gauge group $G$). We may construct tests which enforce these properties.

\begin{example}
    Let $G$ be a finitely generated amenable group with generating set $S$. $G$ acts on $2^{G}$ by translation, i.e. $g[f](x) = f(gx) $. For $\mu$ a probability measure on $2^G$, 
    \[\left(\mu \geq \frac{1}{2|S|}\sum_{g\in S} g_*(\mu) +g^{-1}_*(\mu) \right)\implies \mu \text{ is invariant}\]

    The test on the left hand side is called the \textit{group Laplacian} of $G$.
\end{example}
A $\mu$ weakly passing this test must projectively dominate its translates by group elements. This is a strictly weaker condition than actual invariance, for example it does not preclude the existence of atoms, but rather requires that if the measure has an atom it also has support at all translates of that atom by the group action. The maximum decay rate (or minimum measure assigned to each translate) is exponential in the word length of the translation, and is determined by the constant witnessing the projective domination.

\begin{prop}
    If $2^\N$ is regarded as $2^G$ for $G$ finitely generated and amenable. A neutral measure projectively dominates its translates by elements of $G$ under the natural action of $G$.
\end{prop}

Because they have atoms, neutral measures cannot be invariant under free actions of countable groups. On the other hand:
\begin{example}
    Let $G$ be a finite group acting computably and continuously on $2^{\N}$. Then $G\mu = \frac{1}{|G|}\sum_{g\in G} g\mu$ is a $G$-invariant measure. Furthermore, if $T$ is a universal test, $GT$ is as well (since $GT \geq \frac{1}{|G|}T $), and any neutral measure $\mu$ for $T$ has $G\mu$ neutral for $GT$.
\end{example}
We can, for example, discuss neutral measures which are invariant under the $\Z_2$ action on $2^\Z$ by flipping all bits. Intuitively, we might expect Amenable group actions (on $2^G$ for $G$ countable amenable) to have a similar class of invariant measures, as they still have invariant means. This would require restricting ourselves to tests which do not force a passing measure to be non-invariant.

We have seen that translation and (global) gauge invariance are possible to guarantee with appropriate tests. We will now see something much stronger: that the existence of Gibbs measures for lattice models with finite interaction distance is also a consequence of writing the Gibbs statistics conditions as tests.

\begin{example}
    Fix $T > 0$. Take $X = 2^{\Z^n}$ and $H_1 :{2^{\{-k,..0,...k\}^n}} \rightarrow \R$. For the side-length $2L+1$ cube $K$ around the origin in $\Z^n$ define $H_L: 2^{\Z^n} \rightarrow \R $ by 
    $$ H_L(f) = \sum_{x\in K} H_1(f\upharpoonright_{{\{-k,..0,...k\}^n}})$$

    Let $$Z_L(f) = \sum_{h:h(z)\neq f(z) \implies z\in K} e^{-H_L/T}$$

    where $f$ is a lattice configuration $f:\Z^n\rightarrow \{0,1\}$ and define $G_L: P(2^{\Z^n}) \rightarrow P(2^{\Z^n})$ as the continuous linear completion of its action on Dirac measures, defined here as:

    $$ G_L(\delta f) = \frac{1}{Z_L (f)}\sum_{h:h(z)\neq f(z) \implies z\in K} e^{-H_L(h)/T}\delta h $$    
\end{example}

If $\mu \geq G_L(\mu)$ for all $L$, then $\mu$ is what is called a Gibbs measure for $H_1$ at temperature $T$. These are the measures that represent the expected physical state of a lattice with this energy function. The reason this state is expected is that it maximizes entropy given an expected energy constraint (indeed $G_L$ can be derived from ``freezing" probabilities external to the cube $K$ and maximizing the entropy of cube values subject to an expected energy using the Lagrange multiplier $1/T$).

$G_L$ is also idempotent, and $G_{L+k} G_L = G_{L}G_{L+k} = G_{L+k}$

$G_L$ is a test, and we may define a test $G$ by $G = \sum_L 2^{-L}G_L $. $G(\mu)$ is always a measure, and so if $\G(\mu)\leq \mu$, the two are equal. If $\mu$ is a fixed point of $G$, it is a fixed point of $G^n$, and since $(1-G_L)G^n\rightarrow 0$, $\mu$ is also a fixed point of $G_L$.

\begin{thm}
    Gibbs measures exist at arbitrary positive temperature and, if $H_1$ is a computable function and at computable temperature, form a nonempty $\Pi_1^0$ class.
\end{thm}

\begin{proof}
    By Theorem \ref{thm:mainthm}, there exists a measure passing $G$, which is a fixed point, and therefore a Gibbs measure at temperature $T$. The test $G$ is computable if $T$ is, and thus failure at some finite stage is a computable condition, and the measures which pass are a $\Pi_1^0$ class.
\end{proof}

The most common introductory example of a lattice model is an Ising model, a model of ferromagnetism in which the local state is the up or down directionality of a spin, so a state space $\{-1,1\}$. The relevant $H_1$ here is \[H_1(f) =\sum_{1\leq i\leq n} f(0)(f(e_i)+f(-e_i))\] 
which means that all interactions are with nearest neighbors on the lattice, and the model enjoys a global gauge symmetry given by flipping all spins to their negatives. Such models exhibit phase transitions for $n\geq 2$, which are associated with the \textit{spontaneous breaking} of the global gauge symmetry, in which the single ergodic Gibbs state above a critical temperature splits into a ``mostly up" and ``mostly down" pair of states. Tong \cite{Tong} is an excellent resource for the physical point of view here. At the critical temperature, there are many odd properties exhibited by the associated Gibbs measure, such as the formation of features at all scales \cite{CFT}, and, more quantitatively, multifractality in both the Ising model \cite{IsingMulti}, and related models \cite{DESOUSA2026117517}. More on multifractality will be discussed in the next section \ref{sec:Factor}.

We note here that the existence of Gibbs measures has been proven in this situation and much more difficult situations than this, see for example \cite{Sarig_2003} \cite{Mauldin_Urbański_2001}. The reason we present it here is the relative simplicity of the proof, which we get essentially for free, discounting the cost of proving existence of neutral measures. This suggests that the two are related, and indeed Kolmogorov complexity, which can be used to construct the alternative hypothesis for computable measures on $2^\N$, can be constructed (up to multiplicative error) via thermodynamic means, as in Tadaki's 2008 paper \cite{tadaki2008} and more generally in Baez and Stay \cite{baezstay}. More speculatively, it may be that neutral measures themselves may be constructed by thermodynamic means.

\subsection{Tests Which Factor Through Other Spaces}\label{sec:Factor}
\begin{dfn}
    The \textbf{Dimension Distribution} of a probability measure $\mu$ on a compact metric space $(X,d)$ is 
    $$Ddist(\mu, D) = \mu(d^{-1}[0,D])$$
\end{dfn}

$Ddist(\mu, D)$ is linear in $\mu$ and increasing in $D$. Since it increasing in $D$ and is bounded, it is of bounded variation, and therefore has an associated Lebesgue-Stieltjes measure, which, intuitively speaking, is the probability for randomly drawing a dimension $D$ real when randomly choosing reals using $\mu$. We will call this the ``dimension probability". If $(X,d)$ is finite dimensional, the map taking $\mu$ to its dimension probability is a positive linear total-mass-preserving map from Radon measures on $X$ to measures on $[0,\infty)$

Let $C:2^{\N} \rightarrow [0,1]$ be the map given by $f\rightarrow \sum_nf_{n}2^{-n}$. A measure $\mu$ on Cantor space will push forward via $C$ to a measure $C_*\mu$ on $[0,1]$

\begin{lem}
 If $\mu$ is neutral on $2^\N$, $C_*(\mu)$ is neutral on $[0,1]$
\end{lem}
\begin{proof}
    Suppose $T$ is c.e. a test on $[0,1]$, that is, a map from measures on $[0,1]$ to semimeasures on $[0,1]$ which is defined via enumeration reduction. Define
    $$ A_n=\{\sigma\in 2^{<\N}: \exists a,b\in \Q, C(\sigma) \supset(a,b),  2^n C_*(\mu)\leq TC_*(\mu)[(a,b)]\}$$
    And define $|A_n|$ to be the topological realization of the union of elements of $A_n$. Let $I_k$ be an enumeration of witnessing $(a,b)$ for maximal cylinders in $A_n$, we have
    $$C_*(\mu)(C(|A_n|))\leq 2^{-n}\sum_{k\in\N}TC_*(\mu)(I_k)\leq 2^{-n}TC_*(\mu)([0,1])$$
    and so $\mu(|A_n|)\leq 2^{-n}$. The map $S:\mu \rightarrow \sum_n \chi_{A_n}\mu$ is given by a c.e. test on $2^\N$. $S$ is weakly passed by $\mu$ if and only if $C_*(\mu)$ weakly passes $T$. If $T$ is taken to be a universal c.e. test on $[0,1]$, $\mu$ passes $S$ iff $C_*(\mu)$ is neutral on $[0,1]$, so neutral measures push forward to neutral measures.
\end{proof}

Let $h$ be the inverse of $-xlog(x)-(1-x)log(1-x)$ on $[0,1/2]$. Given a measure $\rho$ on $[0,1]$, we define a measure $P_\rho$ on $2^{\N}$ by

$$P_\rho = \int_{[0,1]} P_{h(x)}d\rho(x)$$
where $P_q$ is defined as in \ref{Pq}. $P_\rho$ is enumerable from $\rho$. So $P_{C_*(\mu)}$ is enumerable from $\mu$, and thus if $\mu$ is Neutral, $\mu >_{p} P_{C_*(\mu)}$.

\begin{lem}
    $P_\rho$ has dimension probability $\rho$
\end{lem}

\begin{proof}

If $2^{\N}$ is viewed as $2^\Z$ then $P_\rho$ is an invariant measure with ergodic decomposition into Bernoulli measures given by the measure $\rho$, after parameterizing these Bernoulli measures by dimension. The pointwise dimension
$$d(x) = \liminf_{r\rightarrow 0}\frac{\log(P_{\rho}(B_r(x)))}{\log(r)} $$
is an invariant function (as $P_\rho$ itself is invariant), thus $d^{-1}(A)$ is an invariant set for $A\subset [0,1]$. The ergodic measure $P_{h(D)}$ assigns measure $1$ to $d^{-1}(\{D\})$ by \ref{entropydensity}, and thus
$$Ddist(P_\rho ,D) = P_\rho(d^{-1}([0,D]) ) = \int_{[0,1]} P_{q}(d^{-1}([0,D])) d\rho(q) = \rho([0,D])$$
\end{proof}

The lemma is the bulk of the work for the following theorem:

\begin{thm}
    The dimension probability of a neutral measure on $2^{\N}$ with the standard metric dominates a neutral measure on $[0,1]$

\end{thm}
\begin{proof}

Let $\mu$ be neutral, $C_*(\mu)$ is also neutral. $P_{C_*(\mu)}$ is a measure on $2^\N$ which is enumerable from $\mu$. Thus $\mu \geq_p P_{C_*(\mu)}$, and so the  of $\mu$ dominates that of $P_{C_*(\mu)}$, which is $C_*(\mu)$.

\end{proof}

Since the dimension probability of a neutral measure dominates a neutral measure, it has many of the same properties of a neutral measure, for example it will have a dense set of atoms, and will projectively dominate Lebesgue measure on the interval. There is a relationship between this fact and the fine dimension spectrum, as defined in Falconer's book \cite{Falconer}. We have the following:

\begin{cor}
    Let $\mu$ be a neutral measure. The Hausdorff dimension of the set of points of pointwise dimension $d(x)= \alpha$ is $\alpha$ for a dense set of $\alpha$. 
\end{cor}
\begin{proof}
    The dimension distribution supports a dense set of atoms. Let $\alpha$ be one such atom. Per Cutler's review \cite{Cutler1993ARO}, the Hausdorff dimension of $d^{-1}(\alpha)$ is at most $\alpha$. We will now show that the dimension of this set is at least $\alpha$ by using the Riesz $s$-energy. As seen in chapter 4, theorem 4.13 of Falconer \cite{Falconer}, if a set supports a probability measure $P$ such that the double integral
    \[I_s(P) = \int\int \frac{1}{d(x,y)^s} dP(x)dP(y) <\infty\]
    that set has Hausdorff dimension at least $s$. We therefore fix $s < \alpha$ and take \[P = \frac{\mu|_{d^{-1}(\alpha)}}{\mu(d^{-1}(\alpha))} \]
    Since we know that the pointwise dimension of every point of $d^{-1}(\alpha)$ is at least $\alpha$ with respect to $P$. We claim that this implies that for $\epsilon>0$ there exists a uniform positive $M$ such that for all cylinders $\sigma$, $P([\sigma])\leq M 2^{-(\alpha-\epsilon)\sigma}$. If there are infinitely many $\sigma$ such that $P([\sigma])>2^{-(\alpha-\epsilon)|\sigma|} $, the collection
    $$ A_n = \{x\in 2^\N: \exists \sigma \text{ extending } x\upharpoonright n, P([\sigma])>2^{-(\alpha-\epsilon)|\sigma|} \}$$
    is a nested collection of nonempty compact sets, and thus nonempty. This contradicts the global bound on pointwise dimension. Thus we may account for the finitely many $\sigma$ not obeying this strict bound via a sufficiently large $M$. We now bound the inner integral as
    $$ \int \frac{1}{d(x,y)^s}dP(x) \leq \sum_{n=1}^\infty 2^{sn}P([y\upharpoonright n])\leq M\sum_{n=1}^\infty 2^{sn}2^{-(\alpha-\epsilon)n}$$
    For $\epsilon$ small enough, this is uniformly summable, and thus the double integral is convergent. This concludes the proof.
\end{proof}

This also establishes that the Hausdorff dimension of the set $d^{-1}((\alpha-\epsilon, \alpha])$ is always $\alpha$ for $\epsilon > 0$. We end with an example that emphasizes that this does not imply that the Hausdorff dimension of $d^{-1}(\alpha)$ for all $\alpha$.

\begin{example}
    Denote by $C_n$ the middle fraction Cantor set with Hausdorff dimension $1- 2^{-n}$ and denote by $\mu_n$ the uniform (coin-flip) measure on that set. Let $I_n$ the interval $[1-2^{-n},1-2^{-n}+2^{-n-2}]$ and take
    \[A =[0,1]\times\{1\}\cup \bigcup_{n} C_n\times I_n \]
    Endow it with the (properly normalized) sum of product measures
    \[\mu = \sum_{n =1}^\infty \mu_n\times \lambda\]
    where $\lambda$ is the Lebesgue measure.
    The pointwise dimension of each point in $C_n\times I_n$ is $2-2^{-n}$, and at each point of $[0,1]\times\{1\}$ is $2$ by a simple estimate. The Hausdorff dimension of this final line segment is, of course, one, but the Hausdorff dimension of the set of points with pointwise dimension $2 - 2^{-n}$ is $2-2^{-n}$ itself.
\end{example}

\section{Bibliography}\label{sec:bib}
\bibliographystyle{plain}
\bibliography{bibliography}

\begin{thebibliography}{10}

\bibitem{baezstay}
John Baez and Mike Stay.
\newblock Algorithmic thermodynamics.
\newblock {\em Computability of the Physical, Mathematical Structures in Computer Science 22}, 2012.

\bibitem{Cutler1993ARO}
Colleen~D. Cutler.
\newblock A review of the theory and estimation of fractal dimension.
\newblock 1993.

\bibitem{daymiller}
Adam Day and Joseph Miller.
\newblock Randomness for non-computable measures.
\newblock {\em Transactions of the American Mathematical Society}, 2013.

\bibitem{DESOUSA2026117517}
Marcos~A.A. {de Sousa}, Francisco~A.B.F. {de Moura}, Anderson~L.R. Barbosa, and Adauto~J.F. {de Souza}.
\newblock Multifractality analysis of phase transition of the two- and three-dimensional xy models.
\newblock {\em Chaos, Solitons and Fractals}, 202:117517, 2026.

\bibitem{DH}
Rodney~G. Downey and Denis~R. Hirschfeldt.
\newblock {\em Algorithmic Randomness and Complexity}.
\newblock Springer New York, NY, 2010.

\bibitem{Falconer}
Kenneth Falconer.
\newblock {\em Fractal Geometry}.
\newblock John Wiley and Sons, Ltd, 2003.

\bibitem{enumerationOG}
Richard~M. Friedberg and Jr. Hartley~Rogers.
\newblock Reducibility and completeness for sets of integers.
\newblock {\em Z. Math. Logik Grundlag.}, 1959.

\bibitem{gacslecturenotes}
Peter Gacs.
\newblock Lecture notes on descriptional complexity and randomness, 2021.

\bibitem{Gacs}
Peter Gács.
\newblock Uniform test of algorithmic randomness over a general space.
\newblock {\em Theoretical Computer Science}, 341(1):91--137, 2005.

\bibitem{levin}
L.~A. Levin.
\newblock Uniform tests of randomness.
\newblock {\em Soviet Math.}, 1973.

\bibitem{lutz2016}
Jack~H. Lutz and Neil Lutz.
\newblock Algorithmic information, plane kakeya sets, and conditional dimension, 2016.

\bibitem{MARTINLOF}
Per Martin-Löf.
\newblock The definition of random sequences.
\newblock {\em Information and Control}, 9(6):602--619, 1966.

\bibitem{Mauldin_Urbański_2001}
R.~Daniel Mauldin and Mariusz Urbański.
\newblock Gibbs states on the symbolic space over an infinite alphabet.
\newblock {\em Israel Journal of Mathematics}, 125(1):93–130, Dec 2001.

\bibitem{CFT}
David~Sénéchal Philippe~Francesco, Pierre~Mathieu.
\newblock {\em Conformal Field Theory}.
\newblock Springer New York, NY, 1997.

\bibitem{Reimann1}
Jan Reimann and Theodore Slaman.
\newblock Measures and their random reals.
\newblock {\em Transactions of the American Mathematical Society}, 2015.

\bibitem{Sarig_2003}
Omri Sarig.
\newblock Existence of gibbs measures for countable markov shifts.
\newblock {\em Proceedings of the American Mathematical Society}, 131(6):1751–1758, Jan 2003.

\bibitem{kakutani}
Kakutani Shizuo.
\newblock A generalization of brouwer's fixed point theorem.
\newblock {\em Duke Mathematical Journal}, 1941.

\bibitem{enumeration}
Theodore~A. Slaman and Mariya~I. Soskova.
\newblock The enumeration degrees: local and global structural interactions.
\newblock {\em Foundations of Mathematics, Contemp. Math}, 2017.

\bibitem{tadaki2008}
Kohtaro Tadaki.
\newblock A statistical mechanical interpretation of algorithmic information theory, 2008.

\bibitem{Tong}
David Tong.
\newblock Statistical field theory.

\bibitem{IsingMulti}
Longfeng Zhao, Wei Li, Chunbin Yang, Jihui Han, Zhu Su, and Yijiang Zou.
\newblock Multifractality and network analysis of phase transition.
\newblock {\em PLOS ONE}, 12(1):1--23, 01 2017.

\end{thebibliography}

\end{document}